\documentclass[10pt]{article}

\title{Intersections of compactly many \\ open sets are open}

\author{Mart\' \i n H\"otzel Escard\'o \\ School of Computer Science \\ University of Birmingham, UK}

\date{27th May 2009}

\usepackage[hidelinks]{hyperref}

\usepackage{times}

\usepackage{url}

\usepackage[english]{babel}

\usepackage{latexsym,amssymb,amsmath,stmaryrd}

\usepackage{theorem,QED}
\newcommand{\theoremshape}{\scshape}

\theoremstyle{change}
\theoremheaderfont{\theoremshape}

\theorembodyfont{\itshape}
\newtheorem{theorem}{Theorem.}[section]
\newtheorem{proposition}[theorem]{Proposition.}
\newtheorem{corollary}[theorem]{Corollary.}
\newtheorem{lemma}[theorem]{Lemma.}
\newtheorem{numbered}[theorem]{}

\theorembodyfont{\normalfont}

\newtheorem{definition}[theorem]{Definition.}

\def\TheWordProof{\theoremshape Proof\enskip}

\newcommand{\old}[1]{}

\newcommand{\waybelow}{\ll}
\newcommand{\eqdef}{\stackrel{{\rm\scriptscriptstyle  def}}{=}}

\newcommand{\Sierp}{\mathbb{S}}

\newcommand{\Opens}{\operatorname{\mathcal{O}}}
\renewcommand{\O}{\Opens}

\newcommand{\comp}{\mathrel{\circ}}
\newcommand{\id}{\operatorname{id}}

\newcommand{\up}{\operatorname{\uparrow}}

\newcommand{\upup}{\mathord{\hbox{\makebox[0pt][l]{\raise .6mm\hbox{$\uparrow$}}$\uparrow$}}}

\newcommand{\Omit}[1]{}

\newcommand{\chron}[1]{}

\begin{document}

\maketitle

\section{Introduction}

By definition, the intersection of finitely many open sets of any
topological space is open.  Nachbin~\cite{nachbin:compactinter}
observed that, more generally, the intersection of compactly many
open sets is open (see Section~\ref{new} for a precise formulation of
this fact). Of course, this is to be expected, because compact sets
are intuitively understood as those sets that, in some mysterious
sense, behave as finite sets. Moreover, Nachbin applied this to obtain
elegant proofs of various facts concerning compact sets in topology
and elsewhere.

\medskip
A simple calculation (performed in Section~\ref{new}) shows that
Nachbin's observation amounts to the well known fact that if a space
$X$ is compact, then the projection map $Z \times X \to Z$ is closed
for every space~$Z$.

\medskip
It is also well known that the converse holds: if a space $X$ has the
property that the projection $Z \times X \to Z$ is closed for every
space~$Z$, then $X$ is compact. We reformulate this as a converse of
Nachbin's observation, and apply this to obtain further elegant proofs
of (old and new) theorems concerning compact sets.

\medskip We also provide a new proof of (a reformulation of) the fact
that a space $X$ is compact if and only if the projection map
$Z \times X \to Z$ is closed for every space~$Z$. This is generalized
in various ways, to obtain new results about spaces of continuous
functions, proper maps, relative compactness, and compactly generated
spaces.

In particular, we give an
intrinsic description of the binary product in the category of compactly generated
spaces in terms of the Scott topology of the lattice of open sets.

\pagebreak
\section{Some characterizations of the notion of compactness} \label{new}

\begin{definition}
Let $Z$ be a topological space and $\{ V_i \mid i \in I\}$ be a family
of open sets of~$Z$. If the index set $I$ comes endowed with a
topology, such a family will be called \emph{continuously indexed}
if whenever $z \in V_i$, there are neighbourhoods $T$ of $z$ and $U$ of $i$ such that $t \in V_u$ for all $t \in T$ and $u \in U$.
This amounts to saying that the graph
$\{ (z,i) \in Z \times I \mid z \in V_i \}$ of the family is open in
the product topology.
\end{definition}
\begin{theorem} \label{compactness}
  A space $X$ is compact if and only if the set $\bigcap_{x \in X}
  V_x$ is open for any continuously indexed family $\{ V_x
  \mid x \in X \}$ of open sets of any space~$Z$.
\end{theorem}
That is, not only are the open sets of any space closed under the
formation of compact intersections, in addition to the postulated
finite intersections, but also this characterizes the notion of
compactness.  In the ``synthetic'' formulation of compactness
developed in~\cite[Chapter 7]{escardo:barbados}, we used continuous
universal quantification functionals, for which function spaces were
required (see Section~\ref{original} below). A related formulation
that avoids the function-space machinery is the following:
\begin{theorem} \label{compactness:universal}
  A space $X$ is compact if and only if for any space $Z$ and any open
  set $W \subseteq Z \times X$, the set $ \{ z \in Z \mid \forall x
  \in X. (z,x) \in W \}$
  is open.
\end{theorem}

This set can be written, more geometrically, as $\{ z \in Z \mid \{z\}
\times X \subseteq W\}$, but the given logical formulation emphasizes
the connection with~\cite{escardo:barbados}.  To
prove~\ref{compactness} and~\ref{compactness:universal}, first observe
that any open set $W \subseteq Z \times X$ gives rise to the
continuously indexed family $\{ V_x \mid x \in X \}$ of open sets of
$Z$ defined by $z \in V_x$ iff $(z,x) \in W$, and that this
construction is a bijection from open sets of $Z \times X$ to
continuously $X$-indexed families of open sets of $Z$.  Moreover, $z
\in \bigcap_{x \in X} V_x$ iff $\forall x \in X.(z,x) \in W$.  Next,
consider the closed set $F \eqdef (Z \times X) \setminus W$ and the
projection $\pi \colon Z \times X \to Z$.  Then $z \in \pi(F)$ iff
$\exists x \in X.  (z,x) \in F$.  By the De Morgan law for existential
and universal quantifiers, $Z \setminus \pi(F) = \{ z \in Z \mid
\forall x \in X.  (z,x) \in W \}$.  It follows that:
\begin{lemma} \label{index:lemma}
  The following are equivalent for spaces $X$ and $Z$.
  \begin{enumerate}
  \item The open sets of $Z$ are closed under continuously $X$-indexed
    intersections.
  \item For any open set $W \subseteq Z \times
    X$, the set $ \{ z \in Z \mid \forall x \in X. (z,x) \in W \}$ is
    open.
  \item The projection $Z \times X \to Z$ is a closed map.
  \end{enumerate}
\end{lemma}
This concludes the proof of~\ref{compactness}
and~\ref{compactness:universal}, because it is well known that
compactness of~$X$ is equivalent to closedness of the projection $Z
\times X \to Z$ for every~$Z$. A self-contained proof of a
generalization of~\ref{compactness:universal}, which doesn't rely on
previous knowledge of the closed-projection characterization of
compactness, is given in Section~\ref{self}.

\pagebreak[4]

\section{A characterization via function spaces} \label{original}

We apply the formulation of compactness given
by~\ref{compactness:universal} to derive the ``synthetic'' formulation
based on function spaces~\cite[Chapter 7]{escardo:barbados}. No
previous knowledge on function-space topologies is required here.

\begin{definition} \label{expo:univ}
  For given spaces $S$ and $X$, we denote by $S^X$ the set of
  continuous maps $X \to S$ endowed with a topology such
  that\footnote{Because the category of continuous maps of topological
    spaces is well pointed, this coincides with the categorical notion
    of exponential.}
\begin{enumerate}
\item the evaluation
map $e \colon S^X \times X \to S$ defined by $e(f,x)=f(x)$ is
continuous,
\item for any space $Z$, if $f \colon Z \times X \to S$ is continuous
  then so is its exponential transpose $\bar{f} \colon Z \to S^X$
  defined by $\bar{f}(z) = (x \mapsto f(z,x))$.
\end{enumerate}
\end{definition}
Such an exponential topology doesn't always exist, but when it does, it
is easily seen to be unique.  Criteria for existence and explicit
constructions can be found in
e.g.~\cite{escardo:heckmann:functionspace} or~\cite[Chapter
8]{escardo:barbados}, or in the extensive set of references contained
therein, but they are not necessary for our purposes.

\begin{definition}
  Let $\Sierp$ be the Sierpinski space with an isolated point $\top$
  (true) and a limit point~$\bot$ (false). That is, the open sets are
  $\emptyset$, $\{\top\}$ and $\{\bot,\top\}$, but not $\{ \bot\}$.
\end{definition}
Then a map $p \colon X \to \Sierp$ is continuous iff $p^{-1}(\top)$ is
open, and a set $U \subseteq X$ is open iff its characteristic map
$\chi_U \colon X \to \Sierp$ is continuous.  Previous proofs of the
following theorem were based on the fact that if the exponential
$\Sierp^X$ exists, then its topology is the Scott topology. The
present proof doesn't require this knowlegde, relying only
on~\ref{compactness:universal} and the universal property of
exponentials given by Definition~\ref{expo:univ}.

\pagebreak[3]
\begin{theorem}
  If the exponential $\Sierp^X$ exists, then the following are equivalent:
  \begin{enumerate}
  \item $X$ is compact.
  \item The universal-quantification functional $A \colon \Sierp^X \to
    \Sierp$ defined by
   \[ A(p) = \top \iff \forall x \in X. p(x) = \top \]
  is continuous.
  \end{enumerate}
\end{theorem}
\begin{proof}
  $(\Downarrow)$: Because the evaluation map $e \colon \Sierp^X \times
  X \to \Sierp$ is continuous, the set $W \eqdef e^{-1}(\top)$ is
  open, and hence $\{ p \in \Sierp^X \mid \forall x \in X.  (p,x)
  \in W \} = 
  A^{-1}(\top)$ is open by compactness of $X$, and therefore $A$ is
  continuous.

  \medskip

  $(\Uparrow)$: Let $Z$ be any space and $W \subseteq Z \times X$ be
  an open set. Because the transpose $w \colon Z \to \Sierp^X$ of $\chi_W
  \colon Z \times X \to \Sierp$ is continuous, so is $A \comp w
  \colon Z \to \Sierp$, and hence $V \eqdef (A \comp
  w)^{-1}(\top)$ is open.  But $z \in V$ iff $A(w(z))=\top$ iff
  $\forall x \in X.  w(z)(x)=\top$ iff $\forall x \in X. (z,x) \in W$.
  This shows that $\{ z \in Z \mid \forall x \in X. (z,x) \in W \}$ is
  open, and hence that $X$ is compact.
\end{proof}

\pagebreak[3]

\section{Generalization of Section~\ref{new}} \label{self}

A proof that compactness of $X$ implies closedness of the projection
$Z \times X \to Z$ for every space~$Z$, which amounts to the
implication~\ref{compactness:universal}($\Rightarrow$), is relatively
easy.  We now formulate and prove a generalization of this implication
for families of compact subsets of the space~$X$.

\begin{definition} \label{compact:indexed:def}
  We say that a family $\{ Q_y \mid y \in Y\}$ of compact subsets
  of~$X$ is continuously indexed\footnote{This is equivalent to
    continuity of the map $y \mapsto Q_y$ when the collection of
    compact sets is endowed with the upper Vietoris topology. For a
    family $V_x$ of open sets of $Z$, however, there isn't a topology
    on the collection of open sets of $Z$ such that continuity of the
    family is equivalent to continuity of the map $x \mapsto V_x$,
    unless $Z$ is an exponentiable space --- see e.g.\ \cite[Corollary
    4.6]{escardo:heckmann:functionspace}.}  by a topological space $Y$
  if for every neighbourhood $U$ of $Q_y$, there is a
  neighbourhood~$T$ of $y$ such that $Q_t \subseteq U$ for all $t \in
  T$. This amounts to saying that the set $\{ y \in Y \mid Q_y
  \subseteq U \}$ is open for every open set~$U\ \subseteq X$.
\end{definition}
The implication \ref{compactness:universal}($\Rightarrow$) is a
special case of the following, considering the space $Y$ with just one
point $y$ and the trivial family $Q_y = X$.
\begin{lemma} \label{compact:indexed}
  Let $\{ Q_y \mid y \in Y\}$ be a continuously indexed family of
  compact sets of a space $X$, let $Z$ be any space, and $W \subseteq
  Z \times X$ be an open set. Then the set
  \[ \{ (z,y) \in Z \times Y \mid \forall_{x \in Q_y}.(z,x) \in W \}  \]
is open.
\end{lemma}
Equivalently,
\[
V_y \eqdef \{ z \in Z \mid \{ z\} \times Q_y \subseteq W\}
\]
is a continuously indexed family of open sets of $Z$.
\begin{proof}
  To show that the set $M \eqdef \{ (z,y) \in Z \times Y \mid \{ z \}
  \times Q_y \subseteq W \}$ is open, we construct, for any pair
  $(z,y) \in M$, open sets $V$ and $T$ with $(z,y) \in V \times T
  \subseteq M$. So assume that $\{ z \} \times Q_y \subseteq W$.  For
  any $x \in Q_y$, we have that $(z,x) \in W$ and hence there are open
  sets $U_x$ and $V_x$ with $(z,x) \in V_x \times U_x \subseteq W$ by
  definition of the product topology.  Then $Q_y \subseteq \bigcup \{
  U_x \mid x \in Q_y\}$, and, by compactness of~$Q_y$, there is a
  finite set $I \subseteq Q_y$ such that already $Q_y \subseteq
  \bigcup \{ U_i \mid i \in I\}$. Let $V \eqdef \bigcap_{i \in I}
  V_i$. Then $V$ is an open neighbourhood of~$z$.  By hypothesis,
  there is an open neighbourhood $T$ of $y$ such that $Q_t \subseteq
  \bigcup \{ U_i \mid i \in I\}$ for all $t \in T$. To show that $V
  \times T \subseteq M$, let $(v,t) \in V \times T$. For any $x \in
  Q_t$, there is $i \in I$ such that $x \in U_i$, and hence $(v,x) \in
  V \times U_i \subseteq V_i \times U_i \subseteq W$, which shows that
  $\{ v \} \times Q_t \subseteq \bigcup_{i \in I} V \times U_i
  \subseteq W$, and therefore that $(v,t) \in M$, as required.
\end{proof}

\medskip

That closedness of the projection $Z \times X \to Z$ for every space
$Z$ implies compactness of $X$, which amounts to the
implication~\ref{compactness:universal}($\Leftarrow$), is less
trivial. Typical proofs apply the characterization of compactness via
cluster points of filters (see e.g.\ the proof of~\cite[Lemma 10.2.1,
page 101]{bourbaki:topology1}).  We offer a proof of a slight
generalization of~\ref{compactness:universal}($\Leftarrow$) that is
closely related to, and inspired by those of
\cite[Lemma~4.4]{escardo:heckmann:functionspace}
and~\cite[Theorem~9.5]{escardo:barbados}. This argument will be reused
later to prove a more general fact about relative compactness
(Section~\ref{relative}).

Recall that a collection $\mathcal{C}$ of open sets is called directed
if for any finite set $\mathcal{S} \subseteq \mathcal{C}$ there is $U
\in \mathcal{C}$ with $\bigcup {\mathcal{S}} \subseteq U$.  Any
collection of open sets can be made directed by adding the finite
unions of its members. Hence a set $Q$ is compact if and only if every
directed open cover of $Q$ has a member that covers $Q$. 
\begin{lemma} \label{projection:nontrivial}
  Let $Q$ be a subset of a space $X$. If for every space $Z$ and any
  open set $W \subseteq Z \times X$, the set $ \{ z \in Z \mid \{ z \}
  \times Q \in W \}$ is open, then $Q$ is compact.
\end{lemma}

\begin{proof}
  Let $\mathcal{C}$ be a directed open cover of~$Q$.

  We first construct a space~$Z$ from $X$ and $\mathcal{C}$: its
  points are the open sets of~$X$, and $V \subseteq Z$ is open iff
  (1)~$U \in V$ and $U \subseteq U' \in Z$ together imply $U' \in V$,
  and (2)~if $\bigcup \mathcal{C} \in V$ then $U \in V$ for some $U
  \in \mathcal{C}$.  Such open sets are readily seen to form a
  topology\footnote{This is like the Scott topology, but defined with
    respect to one particular directed set, rather than all directed
    sets. One cannot use the Scott topology for this proof, as, in
    general, it doesn't give rise to openness of the set $W$
    constructed in the proof --- see e.g.\ \cite[Corollary
    4.6]{escardo:heckmann:functionspace}.}, using the fact that
  $\mathcal{C}$ is directed, and if $U$ is an open subset of a member
  of~$\mathcal{C}$ then $\up U \eqdef \{ U' \in Z \mid U \subseteq
  U'\}$ is clearly
  open. 

  Next, we take $W \eqdef \{ (U,x) \in Z \times X \mid x \in U \}$.
  To show that $W$ is open, let $(U,x) \in W$ and consider two cases.
  (1) $x \in \bigcup \mathcal{C}$: Then $x \in U'$ for some $U' \in
  \mathcal{C}$ with $x \in U'$, and hence $(U,x) \in \up (U \cap U')
  \times (U \cap U') \subseteq W$.  (2) $x \not\in
  \bigcup \mathcal{C}$: Then $U \not\subseteq \bigcup \mathcal{C}$ and
  hence
  $\up U$ is open, and $(U,x) \in \up U \times U \subseteq W$. 

  Finally, by the hypothesis, the set $V \eqdef \{ U \in Z \mid \{ U
  \} \times Q \subseteq W \}$ is open, and clearly $U \in V$ iff $Q
  \subseteq U$.  Hence $ \bigcup \mathcal{C} \in V$ and so some member
  of $\mathcal{C}$ is in $V$, that is, covers $Q$, by construction of
  the topology of $Z$, as required.
\end{proof}

\medskip

For future reference, we summarize part of the above development as
follows:
\begin{lemma} \label{compact:set}
  The following are equivalent for any subset $Q$ of any topological
  space~$X$.
\begin{enumerate}
\item  $Q$ is compact.
\item For every space $Z$, the set $ \{ z \in Z \mid \forall q \in Q.
  (z,q) \in W \} $ is open whenever the set $W \subseteq Z \times X$
  is open.
\item For every space $Z$, the set $ \{ z \in Z \mid \{ z \} \times Q
  \subseteq W \} $ is open whenever the set $W \subseteq Z \times X$
  is open.
\end{enumerate}
\end{lemma}

\pagebreak[3]

\section{Sample ``synthetic'' proofs of old theorems} \label{sample}

We redevelop the synthetic proofs of \cite[Chapter
9]{escardo:barbados} almost literally, but without invoking the
function-space machinery or the lambda-calculus.

\begin{numbered}
  If $X$ is Hausdorff and $Q \subseteq X$ is compact, then $Q$ is
  closed in~$X$.
\end{numbered}
\begin{proof}
  Because $X$ is Hausdorff, the complement $W$ of the diagonal is
  open.  Hence $X \setminus Q = \{ x \in X \mid \forall q \in Q. x \ne
  q \} = \{ x \in X \mid \forall q \in Q. (x,q) \in W\}$ is open by
  Lemma~\ref{compact:set}, and so $Q$ is closed.
\end{proof}

\begin{numbered}
  If $X$ is compact and $F \subseteq X$ is closed then $F$ is compact.
\end{numbered}
\begin{proof}
  We use Lemma~\ref{compact:set}. Let $Z$ be any space
  and $W \subseteq Z \times X$ be open. We have to show that $V \eqdef
  \{ z \in Z \mid \forall x \in F. (z,x) \in W\}$ is open. But $z \in
  V$ iff $\forall x \in X.  x \in F \implies (z,x) \in W$ iff $\forall
  x \in X. x \not\in F \vee (z,x) \in W$.  Hence $V = \{ z \in Z \mid
  \forall x \in X.  (z,x) \in W' \}$ where $W' = (Z \times (X
  \setminus F)) \cup W$, and $V$ is open by compactness of $X$,
  openness of $W'$ and~\ref{compactness:universal}($\Rightarrow$).
\end{proof}

\begin{numbered}
  If $f \colon X \to Y$ is continuous and the set $Q \subseteq X$ is
  compact, then so is $f(Q)$.
\end{numbered}
\begin{proof}
  For any space $Z$ and any open set $W \subseteq Z \times Y$, we have
  that $\{ z \in Z \mid \forall y \in f(Q).(z,y) \in W \} = \{ z \in Z
  \mid \forall x \in Q. (z,f(x)) \in W \}$, which is open by
  compactness of $Q$, because the set $W'$ defined by $(z,x) \in W'$
  iff $(z,f(x)) \in W$ is open by continuity of $f$.
\end{proof}

\begin{numbered}
  If $X$ and $Y$ are compact spaces then so is $X \times Y$.
\end{numbered}
\begin{proof}
  We show that $V \eqdef \{ z \in Z \mid \forall (x,y) \in X \times Y.
  (z,x,y) \in W \}$ is open for any space $Z$ and any open set $W
  \subseteq Z \times X \times Y$. By compactness of $Y$, the set $W'
  \eqdef \{ (z,x) \in Z \times X \mid \forall y \in Y. (z,x,y) \in
  W\}$ is open, and, by compactness of~$X$, the set $\{ z \in Z \mid
  \forall x \in X. (z,x) \in W' \} = V$ is open, as required.
\end{proof}

\medskip

Although we don't need the function-space machinery to develop the
core of topology, we still can use the function-space-free
synthetic approach to prove theorems about function spaces, as we have
done in Section~\ref{original}. Moreover, the abstract definition
of function space as an exponential again suffices.

\begin{numbered}
  If $Y$ is Hausdorff, then so is the exponential $Y^X$ if it exists.
\end{numbered}
\begin{proof}
  The codiagonal of $Y^X$ is $\{ (f,g) \in Y^X \times
  Y^X \mid \exists x \in X. f(x) \ne g(x) \} = \bigcup_{x \in X} \{
  (f,g) \in Y^X \times Y^X \mid f(x) \ne g(x) \} $, which is a union
  of open sets, because $W \subseteq Y^X \times Y^X$ defined
  by $(f,g) \in W$ iff $f(x) \ne g(x)$ is open, using openness
  of the codiagonal of $Y$ and continuity of the evaluation map $Y^X
  \times X \to Y$.
\end{proof}

For the proof of the following dual proposition, recall that a space
is discrete iff its diagonal is open.
\begin{numbered}
  If $X$ is compact and $Y$ is discrete, then the exponential $Y^X$ is
  discrete if it exists.
\end{numbered}
\begin{proof}
  The diagonal of $Y^X$ is $\{ (f,g) \in Y^X \times Y^X \mid \forall x
  \in X. f(x) = g(x) \}$, which is open by compactness of $X$, because
  the set $W \subseteq Y^X \times Y^X \times X$ defined by
  $(f,g,x) \in W$ iff $f(x) = g(x)$ is open, using openness of
  the diagonal of $Y$ and continuity of the evaluation map.
\end{proof}

As discussed above, these last two propositions don't require an
intrinsic description of the topology of~$Y^X$. A partial description
is given by the following:
\begin{numbered}
  If the exponential $Y^X$ exists, and if $Q \subseteq X$ is compact
  and $V \subseteq Y$ is open, then the set $N(Q,V) \eqdef \{ f \in
  Y^X \mid f(Q) \subseteq V \}$ is open.
\end{numbered}
\begin{proof}
  $f \in N(Q,V)$ iff $\forall q \in Q. f(q) \in V$. The result then
  follows from the fact that $W \subseteq Y^X \times X$
  defined by $(f,x) \in W$ iff $f(x) \in V$ is open, using continuity
  of the evaluation map.
\end{proof}

\section{Proper maps}
\label{proper}

\medskip Recall that a continuous map $f \colon X \to Y$ is called proper if
the product map \[ \id_Z \times f \colon Z \times X \to Z \times Y \]
is closed for every space $Z$, where $\id_Z \colon Z \to Z$ is the
identity map~\cite{bourbaki:topology1}.
\begin{theorem} \label{proper:charac}
The following are equivalent for any continuous map $f \colon X \to Y$.
\begin{enumerate}
\item \label{proper:1} $f$ is proper.
\item \label{proper:2} For every space $Z$ and every open set $W
  \subseteq Z \times X$, the set \[ \{ (z,y) \in Z \times Y \mid \{z\}
  \times f^{-1}\{y\} \subseteq W\} \] is open.
\item \label{proper:3} $f$ is closed and the set $f^{-1}(Q)$ is compact for
  every compact set $Q \subseteq Y$.
\item \label{proper:4} $f$ is closed and the set $f^{-1}\{y\}$ is compact for
  every point $y \in Y$.
\item \label{proper:5} $\{ f^{-1}\{y\} \mid y \in Y \}$ is a
  continuously indexed family of compact sets of $X$.
\end{enumerate}
\end{theorem}
We first refomulate closedness in terms of open sets.  By taking
complements, a continuous map $g \colon A \to B$ is closed iff for
every open set $U\subseteq A$, the set $B \setminus g(A \setminus U)$
is open. But an easy calculation shows that this set is $ \{ b \in B
\mid g^{-1}\{b\}\subseteq U \}$. This proves:
\begin{lemma} \label{reformulation}
A continuous map $g \colon A \to B$ is
  closed if and only if for every open set $U \subseteq A$, the set $\{
  b \in B \mid g^{-1}\{b\} \subseteq U \}$ is open.
\end{lemma}

\medskip
\begin{proof} of Theorem~\ref{proper:charac}.

  \medskip
  $(\ref{proper:1}) \Leftrightarrow (\ref{proper:2})$: Calculate
  that $(\id_Z \times f)^{-1}\{(z,y)\} = \{z\} \times f^{-1}\{y\}$ and then
  apply Lemma~\ref{reformulation} to $g = \id_Z \times f$.

  \medskip

  ($\ref{proper:1},\ref{proper:2}) \Rightarrow (\ref{proper:3}$):
  Considering the case in which $Z$ is the one-point space, we see
  that any proper map is closed.
  To show that $f^{-1}(Q)$ is compact, let $Z$ be any space and $W
  \subseteq Z \times X$ be an open set. Then the set $T \eqdef \{
  (z,y) \mid \{z\} \times f^{-1}\{y\} \subseteq W\}$ is open by
  hypothesis, and hence the set $U \eqdef \{ z \in Z \mid \{ z \}
  \times Q \subseteq T \} $ is open by Lemma~\ref{compact:set}. But $z
  \in U$ iff $(z,y) \in T$ for all $y \in Q$, iff $\{z\} \times
  f^{-1}\{y\} \subseteq W$ for all $y \in Q$, iff $\{z\} \times
  f^{-1}(Q) \subseteq W$.  Because $Z$ and $W$ are arbitrary, a second
  application of Lemma~\ref{compact:set} shows that $f^{-1}(Q)$ is
  compact, as required.

  \medskip ($\ref{proper:3}) \Rightarrow (\ref{proper:4}$): Singletons
  are compact.

  \medskip
  ($\ref{proper:4}) \Rightarrow (\ref{proper:5}$): 
  By Lemma~\ref{reformulation} applied to $g=f$, the set $\{ y \in Y
  \mid f^{-1}\{y\} \subseteq U \}$ is open for every $y \in Y$ and
  every open set $U
  \subseteq X$. 

  \medskip ($\ref{proper:5}) \Rightarrow (\ref{proper:2}$): This
  follows directly from Lemma~\ref{compact:indexed}.
\end{proof}

The above characterizations~(\ref{proper:3}) and~(\ref{proper:4}) of
propriety are of course well known. The development of synthetic
proofs was left as an exercise in~\cite{escardo:barbados}.
Characterization~(\ref{proper:2}) is clearly just a reformulation of
the definition using the language of open sets.
Formulation~(\ref{proper:5}) seems to be new.

We conclude this section with a well known fact about proper maps.
\begin{numbered}
 If $X$ is compact and $Y$ is Hausdorff, then any continuous map $f
 \colon X \to Y$ is proper.
\end{numbered}
\begin{proof}
 To apply the characterization~\ref{proper:charac}(\ref{proper:2}),
 let $Z$ be any space and $W \subseteq Z \times X$ be open. We have
 to show that $T=\{ (z,y) \in Z \times Y \mid \{z\} \times
 f^{-1}\{y\} \subseteq W\}$ is open. Now $(z,y) \in T$ iff $\forall x
 \in X$, $f(x)=y$ implies $(z,x) \in W$, iff $\forall x \in X$, $f(x)
 \ne y$ or $(z,x) \in W$
\end{proof}

\section{Relative compactness} \label{relative}

For some topological questions regarding local compactness and
function spaces, it is fruitful to consider the domain-theoretic
way-below relation on open sets~\cite{gierz:domains}. Again in a
context pertaining to function spaces, Escard\'o, Lawson and
Simpson~\cite{escardo:lawson:simpson} found it profitable to
generalize this to arbitrary subsets of topological spaces.

For subsets $S$ and $T$ of a topological space $X$, we define
\begin{eqnarray*}
S \waybelow T & \iff & \text{every cover of $T$ by open sets of $X$} \\
              &      & \text{has a finite subcollection that covers~$S$.}
\end{eqnarray*}
In this case one says that $S$
is \emph{way below} $T$, or \emph{compact relative} to $T$.
Then it is immediate that a set is
compact iff it is compact relative to itself.
\newcommand{\subsetint}{\Subset}
We also define
\begin{eqnarray*}
S \subsetint T & \iff & S  \subseteq T^\circ.
\end{eqnarray*}
The following was formulated as~\cite[Lemma
4.2]{escardo:lawson:simpson}:
\begin{numbered}\label{equivalence}
  Let $X$ and $Y$ be topological spaces.
\begin{enumerate}
\item \label{equiv:i} If $F\waybelow X$ is closed, then $F$
  is compact.
\item \label{equiv:ii} If $X$ is Hausdorff and $S \waybelow T$ holds in $X$,
  then $\overline S \subseteq T$.
\item \label{equiv:iii} If $f \colon X\to Y$ is continuous and $S\waybelow
  T$ in $X$, then $f(S)\waybelow f(T)$ holds in $Y$.
\item \label{equiv:iv} If $S \waybelow T$ in $X$ and $A
  \waybelow B$ in $Y$, then $S \times A \waybelow T \times B$
  holds in $X \times Y$.
\item \label{equiv:v}
  If $W \subseteq Y \times X$ is open and $S \waybelow T$ holds in $X$, then
  \[ \{y \in Y \mid
  \{y\} \times T \subseteq W\} \subsetint \{y \in Y \mid \{y\} \times
  S \subseteq W\}.
  \]
\end{enumerate}
\end{numbered}

Assertion (\ref{equiv:i}) generalizes the fact that a closed subset of
a compact space is compact, (\ref{equiv:ii})~the statement that a
compact subset of a Hausdorff space is closed, (\ref{equiv:iii})~the
fact that continuous maps preserve compactness, and
(\ref{equiv:iv})~the Tychonoff theorem in the finite case.
In this section we prove (\ref{equiv:v}) and a converse,
generalizing~\ref{compactness:universal} and the development of
Section~\ref{self}, and use this to derive
(\ref{equiv:i})--(\ref{equiv:iv}), generalizing the development of
Section~\ref{sample}.  \pagebreak[3]

\pagebreak[4]
\begin{theorem} \label{relative:universal}
  The following are equivalent for any two subsets $S$ and $T$ of a
  topological space~$X$.
  \begin{enumerate}
  \item \label{thm:1} $S \waybelow T$.
  \item \label{thm:2} For every space $Z$ and every open set $W \subseteq Z \times X$,
  \[ \{z \in Z \mid
  \{z\} \times T \subseteq W\} \subsetint \{z \in Z \mid \{z\} \times S
  \subseteq W\}.
  \]
\item \label{thm:3} For every space $Z$, every $z \in Z$ and every
  open set $W \subseteq Z \times X$, \[ \{z\} \times T \subseteq W
  \implies \text{$V \times S \subseteq W$ for some neighbourhood $V$ of $z$}.\]
\item \label{thm:4} For every space $Z$ and all $M,N \subseteq Z
  \times X$,
       \[
       M \subsetint N \implies \{z \in Z \mid \{z\} \times T
       \subseteq M\} \subsetint \{z \in Z \mid \{z\} \times S \subseteq
       N\}.
       \]
  \end{enumerate}
\end{theorem}
\pagebreak[3]
\begin{proof}
  $(\ref{thm:2}) \Leftrightarrow (\ref{thm:3})$: By definition of
  interior.
  $(\ref{thm:2}) \Leftrightarrow (\ref{thm:4})$: Consider $W=N^\circ$
  in one direction and $M=N=W$ in the other.

 \medskip

 $(\ref{thm:1}) \Rightarrow (\ref{thm:3})$: Assume that $\{ z \}
 \times T \subseteq W$. Then for any $t \in T$, we have that $(z,t)
 \in W$ and hence there are open sets $U_t$ and $V_t$ with $(z,t) \in
 V_t \times U_t \subseteq W$. Because $T \subseteq \bigcup_{t \in T}
 U_t$ and $S \waybelow T$, there is a finite set $I \subseteq T$ such
 that $S \subseteq \bigcup_{i \in I} U_i$. Then $V \eqdef \bigcap_{i
   \in I} V_i$ is open and $z \in V$.  To show that $V \times S
 \subseteq W$, let $(v,s) \in V \times S$. Because $s \in S \subseteq
 \bigcup_{i \in I} U_i$, there is $j \in I$ such that $s \in U_j$, and
 because $V = \bigcap_{i \in I} V_i$ we have that $v \in V_j$. Hence
 $(v,s) \in V_j \times U_j \subseteq W$, as required.

 \medskip

 $(\ref{thm:3}) \Rightarrow (\ref{thm:1})$. To show that $S \waybelow
 T$, let $\mathcal{C}$ be a directed open cover of~$T$.  We have to
 conclude that $S \subseteq U$ for some $U \in \mathcal{C}$.
 We first construct a space $Z$ from $X$ and $\mathcal{C}$, and an
 open set $W \subseteq Z \times X$ as in the proof
 of~\ref{projection:nontrivial}.  Because $T \subseteq \bigcup
 \mathcal{C}$, we have that $\{\bigcup \mathcal{C}\} \times T
 \subseteq W$. Hence, by the hypothesis, $V \times S \subseteq W$ for
 some neighbourhood $V$ of $\bigcup \mathcal{C}$, which may be assumed
 to be open. By construction of the topology of $Z$, we have that $U
 \in V$ for some $U \in \mathcal{C}$. To show that $S \subseteq U$,
 concluding the proof, let $s \in S$. Then $(U,s) \in V \times S
 \subseteq W$, and hence $s \in U$, as required.
\end{proof}

\medskip

Notice that~\ref{compactness:universal} follows directly from
Theorem~\ref{relative:universal}$(\ref{thm:1} \Leftrightarrow
\ref{thm:2})$, because a set is open iff it is contained in its
interior.  Observe also that the implication $(\ref{thm:1})
\Rightarrow (\ref{thm:3})$ amounts to saying that if the relation
$\{y\} \times T \subseteq W$ holds, and if we make $T$ significantly
smaller by passing to a set way below, then we can make $\{y\}$
significantly bigger by passing to a whole neighbourhood so that the
relation will still hold. 
We now apply Theorem~\ref{relative:universal} to generalize some of
the proofs of Section~\ref{sample}.
\begin{numbered}
  If $X$ is Hausdorff and $S \waybelow T$, then $\overline S \subseteq T$.
\end{numbered}
\begin{proof}
  Because the complement $W \subseteq X \times X$ of the diagonal is
  open as $X$ is Hausdorff,
  Theorem~\ref{relative:universal}$(\ref{thm:1} \Rightarrow
  \ref{thm:2})$ shows that $X \setminus T = \{ x \in X \mid
  \{x\}\times T \subseteq W\} \subsetint \{ x \in X \mid \{x\}\times S
  \subseteq W\} = X \setminus S$, and hence $\overline S \subseteq T$.
\end{proof}

\begin{numbered}
  If $F$ is closed in $X$ and $F \waybelow X$, then $F$ is
  compact.
\end{numbered}
\begin{proof}
  Let $Z$ be any space and $W \subseteq Z \times X$ be open. Then $W'
  = (Z \times (X \setminus F)) \cup W$ is also open, and
  Theorem~\ref{relative:universal}$(\ref{thm:2} \Rightarrow
  \ref{thm:1})$ gives $M \eqdef \{ z \in Z \mid \{ z \} \times X
  \subseteq W' \} \subsetint N \eqdef \{ z \in Z \mid \{ z \} \times F
  \subseteq W' \}$. But one readily checks that $M$ and $N$ are equal
  to $\{ z \in Z \mid \forall x \in F. (z,x) \in W \}$, and hence,
  being contained in its own interior, this set is open. Because the
  space $Z$ and the open set $W \subseteq Z \times X$ are arbitrary,
  the desired result follows from Lemma~\ref{compact:set}.
\end{proof}

\begin{numbered}
  If $f \colon X \to Y$ is continuous and $S \waybelow T$ in $X$,
  then $f(S) \waybelow f(T)$ holds in $Y$.
\end{numbered}
\begin{proof}
  Let $Z$ be a space, $W \subseteq Z \times Y$ be open, and assume
  that $\{z\} \times f(T) \subseteq W$. Then $W' \eqdef (\id_Z \times
  f)^{-1}(W) = \{ (z,x) \in Z \times X \mid (z,f(x)) \in W \}$ is also
  open by continuity of $f$, and $\{z\} \times T \subseteq W'$. By
  Theorem~\ref{relative:universal}$(\ref{thm:1} \Rightarrow
  \ref{thm:3})$, there is a neighbourhood $V$ of $z$ with $V \times S
  \subseteq W'$. Hence $V \times f(S) \subseteq W$. Because the space
  $Z$, the open set $W \subseteq Z \times Y$ and the point $z \in Z$
  are arbitrary, Theorem~\ref{relative:universal}$(\ref{thm:3}
  \Rightarrow \ref{thm:1})$ shows that $f(S) \waybelow f(T)$, as
  required.
\end{proof}

\begin{numbered}
  If $S \waybelow T$ in $X$ and $A \waybelow B$ in $Y$,
  then $S \times A \waybelow T \times B$ holds in $X \times
  Y$.
\end{numbered}
\begin{proof}
  Let $Z$ be a space and let $M,N \subseteq Z \times X \times Y$
  with $M \subsetint N$.  Then, by two successive applications of
  Theorem~\ref{relative:universal}$(\ref{thm:1} \Rightarrow
  \ref{thm:4})$, we first have that
  \begin{gather*} M' \eqdef \{ (z,x) \in Z
  \times X \mid \{ (z,x)\} \times B \subseteq M \} \\ \subsetint  N'
  \eqdef \{ (z,x) \in Z \times X \mid \{ (z,x)\} \times A
  \subseteq N\}
  \end{gather*} and then that $M'' \eqdef \{ z \in Z \mid \{z\}
  \times T \subseteq M' \} \subsetint N'' \eqdef \{ z \in Z \mid \{z\}
  \times S \subseteq N' \}$. But one readily checks that $M'' = \{ z
  \in Z \mid \{ z \} \times T \times B \subseteq M \}$ and $N'' = \{ z
  \in Z \mid \{ z \} \times S \times A \subseteq N \}$.  Because the
  space $Z$ and the sets $M,N \subseteq Z \times X \times Y$ are
  arbitrary, the result follows from
  Theorem~\ref{relative:universal}$(\ref{thm:4} \Rightarrow
  \ref{thm:1})$.
\end{proof}

\section{Compactly generated spaces} \label{compactly}

In this section we assume familiarity with the notions and results
developed in~\cite{escardo:lawson:simpson} and with domain
theory~\cite{gierz:domains}.

Let $\mathcal{E}$ be the class of all spaces that are exponentiable in
the category of topological spaces, and $\mathcal{C} \subseteq
\mathcal{E}$ be any productive class of spaces. If $\mathcal{C}$
consists of the compact Hausdorff spaces, then the
$\mathcal{C}$-generated spaces (or $\mathcal{C}$-spaces for short) are
known as the \emph{compactly generated spaces}.

The categorical product in the category of $\mathcal{C}$-spaces is
given by the $\mathcal{C}$-coreflection of the topological product:
\linebreak[3] $X \times_{\mathcal{C}} Y = \mathcal{C}(X \times Y)$.
Recall that the $\mathcal{C}$-coreflection $\mathcal{C} X$ of a
topological space $X$ is obtained by keeping the same points and
suitably refining the given topology of~$X$.
By~\cite[Theorem~5.4]{escardo:lawson:simpson}, we know that $X
\times_\mathcal{C} Y = X \times_\mathcal{E} Y$ for all
$\mathcal{C}$-spaces~$X$ and~$Y$. That is, the $\mathcal{C}$-product
doesn't depend on~$\mathcal{C}$, even though the
$\mathcal{C}$-coreflection does. We were thus led to ask whether there
is an instrinsic characterization of the $\mathcal{C}$-product
\cite[Problem 9.3]{escardo:lawson:simpson}.  We now develop an answer
to this question, formulated as Theorem~\ref{prod:charac} below. We
know that the Sierpinski space is a $\mathcal{C}$-generated space if
and only if the generating class~$\mathcal{C}$ includes a space in
which not every open set is closed~\cite[Lemma
4.6(ii)]{escardo:lawson:simpson}. In particular, the Sierpinski space
is $\mathcal{E}$-generated.  \pagebreak[3]
\begin{lemma} \label{fromels}
  Assume that the Sierpinski space is $\mathcal{C}$-generated.  For a
  $\mathcal{C}$-generated space~$X$, let \[ \O_\mathcal{C} X \] be the
  lattice of open sets of $X$ endowed with the topology that makes the
  bijection  $U \mapsto \chi_U \colon \O_\mathcal{C} X
  \to \Sierp^X$ into a homeomorphism, where the exponential is
  calculated in the category $\mathrm{Top}_\mathcal{C}$ of
  $\mathcal{C}$-spaces.
  \begin{enumerate}
  \item \label{lem:1} The topology of $\O_\mathcal{C} X$ is finer than
    the Scott topology.
  \item \label{lem:2} The topology of $\O_\mathcal{C} X$ coincides
    with the Scott topology if $\mathcal{C}$ generates all compact
    Hausdorff spaces.
  \item \label{lem:3} A set $W \subseteq Y \times_\mathcal{C} X$ is
    open if and only if its transpose $w: Y \to \O_\mathcal{C} X$
    defined by $w(y) = \{ x \in X \mid (y,x) \in W \}$ is continuous.
  \item \label{lem:4} The set $\{ (U,x) \in \O_\mathcal{C} X
    \times_\mathcal{C} X \mid x \in U\}$ is open in the
    $\mathcal{C}$-product.
  \end{enumerate}
\end{lemma}
\begin{proof}
  (\ref{lem:1}): \cite[Theorem 5.15]{escardo:lawson:simpson}.
  (\ref{lem:2}): \cite[Corollary 5.16]{escardo:lawson:simpson}.
  (\ref{lem:3}): By definition of exponential transpose.
  (\ref{lem:4}): Its transpose is the identity of $\O_\mathcal{C} X$.
\end{proof}

If $\mathcal{C}$ doesn't generate all compact Hausdorff spaces, the
second item doesn't necessarily hold. For example, if $\mathcal{C}$ is
a singleton consisting of the one-point compactification of the
discrete natural numbers (known as the ``generic convergent
sequence''), then a space is $\mathcal{C}$-generated if and only if it
is sequential, and for a sequential space $X$ we have that
$\mathcal{U} \subseteq \O_\mathcal{C} X$ is open if and only if it is
upwards closed and inaccessible by unions of countable directed sets.
If $X$ is a Lindel\"of space, as is the case if $X$ is a QCB space,
this does coincide with the Scott topology, but, in general, this is
strictly finer than the Scott topology.  The following holds without
any assumption on $\mathcal{C}$ other than that it is contained in
$\mathcal{E}$ and that it is productive.
\begin{theorem} \label{prod:charac}
  If $X$ and $Y$ are $\mathcal{C}$-spaces, then the following are
  equivalent for any set $W \subseteq Y \times X$.
  \begin{enumerate}
  \item \label{prod:1} $W$ is open in $Y \times_{\mathcal{C}} X$.
  \item \label{prod:2}
    \begin{enumerate}
    \item For each $y \in Y$, the set $U_y \eqdef \{ x \in X \mid
      (y,x) \in W\}$ is open, and
    \item for each Scott open set $\mathcal{U} \subseteq \O X$, the
      set $V_\mathcal{U} \eqdef \{ y \in Y \mid U_y \in \mathcal{U} \}$ is
      open.
    \end{enumerate}
  \item \label{prod:3}
    \begin{enumerate}
    \item For each $x \in X$, the set $V_x \eqdef \{ y \in Y \mid
      (y,x) \in W\}$ is open, and
    \item for each Scott open set $\mathcal{V} \subseteq \O Y$, the
      set $U_\mathcal{V} \eqdef \{ x \in X \mid V_x \in \mathcal{V} \}$ is
      open.
    \end{enumerate}
  \end{enumerate}
\end{theorem}
\begin{proof}
  We prove $(\ref{prod:1}) \Leftrightarrow (\ref{prod:2})$. A proof of
  $(\ref{prod:1}) \Leftrightarrow (\ref{prod:3})$ is obtained via the
  canonical homeomorphism $X \times_\mathcal{C} Y \cong Y
  \times_\mathcal{C} X$.

  $(\ref{prod:1}) \Rightarrow (\ref{prod:2})$: As we have already
  discussed, if $W$ is open in the $\mathcal{C}$-product, then it is
  also open in the $\mathcal{E}$-product.  Because the Sierpinski
  space $\Sierp$ is an $\mathcal{E}$-space, its transpose $w \colon Y
  \to \O_\mathcal{E} X$ defined in the previous lemma is continuous.
  Then one readily checks that $w(y) = U_y$ and $w^{-1}(\mathcal{U}) =
  V_\mathcal{U}$, which shows that $U_y$ and $V_\mathcal{U}$ are open,
  as required.

  $(\ref{prod:2}) \Rightarrow (\ref{prod:1})$: By the hypothesis, the
  map $w \colon Y \to \O_\mathcal{E} X$ given by $w(y) = U_y$ is well
  defined and continuous.  But one readily checks that this is the
  transpose of $W$ defined in Lemma~\ref{fromels}, and hence $W$ is
  open in the $\mathcal{E}$-product, and therefore in the
  $\mathcal{C}$-product, as required.
\end{proof}

We now return to the subject of compactness. We henceforth assume that
the Sierpinski space is $\mathcal{C}$-generated.
\begin{theorem}
  The following are equivalent for any subset $Q$ of a
  $\mathcal{C}$-space $X$.
  \begin{enumerate}
  \item \label{ccomp:1} The set $\{ U \in \O_\mathcal{C} X \mid Q
    \subseteq U \}$ is open.
  \item \label{ccomp:2} For every $\mathcal{C}$-space $Y$, and every
    open set $W \subseteq Y \times_\mathcal{C} X$, the set \[\{ y \in Y
    \mid \{y\} \times Q \subseteq W\}\] is open.
  \item \label{ccomp:3} The universal-quantification functional $A_Q
    \colon \Sierp^X \to \Sierp$ defined by
   \[ A_Q(p) = \top \iff \forall x \in Q. p(x) = \top \]
   is continuous, where the exponential is calculated in
   $\mathrm{Top}_\mathcal{C}$.
  \end{enumerate}
\end{theorem}
\begin{proof}
  $(\ref{ccomp:1}) \Rightarrow (\ref{ccomp:2})$: One readily checks
  that the set $\{ y \in Y \mid \{y\} \times Q \subseteq W\}$ is the
  same as $V_\mathcal{U}$ in Theorem~\ref{prod:charac}(\ref{prod:2})
  for the choice $\mathcal{U}=\{ U \in \O_\mathcal{C} X \mid Q
  \subseteq U\}$.

  $(\ref{ccomp:2}) \Rightarrow (\ref{ccomp:3})$ Because the evaluation
  map $e \colon \Sierp^X \times_\mathcal{C} X \to \Sierp$ is
  continuous, the set $W \eqdef e^{-1}(\top)$ is open, and hence the
  set $\{ p \in \Sierp^X \mid \{ p \} \times Q \subseteq W \} = \{ p
  \in \Sierp^X \mid \forall x \in Q.  p(x) = \top \} = A^{-1}(\top)$
  is open by the hypothesis, and therefore $A_Q$ is continuous.

  $(\ref{ccomp:3}) \Rightarrow (\ref{ccomp:1})$: The set $\{ U \in
  \O_\mathcal{C} X \mid Q \subseteq U \}$ is the inverse image of
  $\{\top\}$ for the composite $A_Q \comp (U \mapsto \chi_U) \colon
  \O_C X \to \Sierp^X \to \Sierp$.
\end{proof}

\begin{definition}
  When these equivalent conditions hold, we say that $Q$ is
  \emph{$\mathcal{C}$-compact}.
\end{definition}
For example, it follows from the above observations that if the class
$\mathcal{C}$ is a singleton consisting of the generic convergent
sequence, then a $\mathcal{C}$-generated space (i.e.\ a sequential
space) is $\mathcal{C}$-compact if and only if every countable open
cover has a finite subcover. However, for compactly generated spaces,
the same notion of compactness is obtained, as shown by the next
proposition. We first formulate an immediate consequence of the above
theorem.
\begin{corollary}
  A $\mathcal{C}$-space $X$ is $\mathcal{C}$-compact if and only if
  the projection $Y \times_\mathcal{C} X \to Y$ is closed for every
  $\mathcal{C}$-space $Y$.
\end{corollary}
\begin{proof}
  Use the De Morgan laws as in Section~\ref{new}.
\end{proof}

\begin{proposition}
  Any compact set is $\mathcal{C}$-compact. If the class $\mathcal{C}$
  generates all compact Hausdorff spaces, the converse holds.
\end{proposition}
\begin{proof}
  If $Q$ is compact subset of a $\mathcal{C}$-space $X$, then $\{ U
  \in \O X \mid Q \subseteq U \}$ is Scott open by definition of the
  Scott topology, and hence open in $\O_\mathcal{C} X$ by
  Lemma~\ref{fromels}.  Conversely, if $Q$ is $\mathcal{C}$-compact
  and the hypothesis holds, then $\{ U \in \O X \mid Q \subseteq U \}$
  is Scott open by Lemma~\ref{fromels}, and hence compact by
  definition of the Scott topology.
\end{proof}

Thus, even though $Y \times_\mathcal{C} X$ has a greater (and somewhat
mysterious) supply of open sets than $Y \times X$, it is still the
case that if $Q$ is compact then for every open set $W \subseteq Y
\times_\mathcal{C} X$, the set $\{ y \in Y \mid \{y\} \times Q
\subseteq W\}$ is open.

\begin{definition}
  We say that a $\mathcal{C}$-space $X$ is
  \emph{$\mathcal{C}$-Hausdorff} if its diagonal is closed in $X
  \times_\mathcal{C} X$, and that it is \emph{$\mathcal{C}$-discrete}
  if its diagonal is open in $X \times_\mathcal{C} X$.
\end{definition}
If a $\mathcal{C}$-space is Hausdorff (resp.\ discrete) then it is
$\mathcal{C}$-Hausdorff (resp.\ -discrete), because the
$\mathcal{C}$-product has a topology finer than the topological
product. There must be $\mathcal{C}$-Hausdorff spaces which are not
Hausdorff, but I doubt that this holds for discreteness.

We have developed enough ideas and techniques to routinely develop
proofs of the following, and hence we omit them:
\begin{proposition} Let $X$ and $Y$ be $\mathcal{C}$-spaces.
  \begin{enumerate}
  \item If $X$ and $Y$ are $\mathcal{C}$-compact, then so is $X
    \times_\mathcal{C} Y$.

    \emph{This potentially fails if one replaces
      $\mathcal{C}$-compactness by topological compactness, because
      the $\mathcal{C}$-product has a topology finner than the
      topological product.}

  \item If $f \colon X \to Y$ is continuous and $Q \subseteq X$ is
    $\mathcal{C}$-compact, then so is $f(Q)$.

  \item If $X$ is $\mathcal{C}$-Hausdorff and $Q \subseteq X$ is
    $\mathcal{C}$-compact, then $Q$ is closed.

    \emph{Notice that this is stronger than the statement that a
      compact subspace of a Hausdorff $\mathcal{C}$-space is closed,
      as it has weaker hypotheses.}

  \item If $F \subseteq X$ is closed and $X$ is $\mathcal{C}$-compact,
    then so is $F$.

  \item If $Y$ is $\mathcal{C}$-Hausdorff, then so is the exponential
    $Y^X$.

\item If $X$ is $\mathcal{C}$-compact and $Y$ is
  $\mathcal{C}$-discrete, then the exponential $Y^X$ is
  $\mathcal{C}$-discrete.

\item If $Q \subseteq X$ is $\mathcal{C}$-compact and $V \subseteq Y$
  is open, then $\{ f \in Y^X \mid f(Q) \subseteq V \}$ is open.
  \end{enumerate}
\end{proposition}
Nb.\ We can define the \emph{$\mathcal{C}$-Isbell topology} on the set
of continuous maps $X \to Y$ as the usual Isbell topology, replacing
Scott openness by openness in $\O_\mathcal{C} X$. It is easy to see
that the exponential topology is finer than the $\mathcal{C}$-Isbell
topology. 

\newcommand{\Scott}{\operatorname{\Sigma}}

\medskip

We now develop another application of Theorem~\ref{prod:charac}.
It is well-known that the (full and faithful) functor $\Scott \colon
\mathrm{DCPO} \to \mathrm{Top}$ from the category of dcpos to
topological spaces, that endows a dcpo with its Scott topology and
acts identically on maps, fails to preserve finite
products~\cite{gierz:domains}. By~\cite[Theorem
4.7]{escardo:lawson:simpson}, we know that dcpos under the Scott
topology are compactly generated. Thus, if every compactly generated
space is a $\mathcal{C}$-space then $\Scott$ factors through the
category $\mathrm{Top}_\mathcal{C}$ of $\mathcal{C}$-spaces.
This is the case, for instance, if $\mathcal{C}=\mathcal{E}$ or
$\mathcal{C}$ consists of all compact Hausdorff spaces or of all
locally compact spaces.
\begin{theorem}
  If $\mathcal{C}\subseteq\mathcal{E}$ generates all compact Hausdorff
  spaces, then the functor $\Scott \colon \mathrm{DCPO} \to
  \mathrm{Top}_\mathcal{C}$ preserves finite products.
\end{theorem}
\begin{proof}
  Let $D$ and $E$ be dcpos. By Theorem~\ref{prod:charac}$(\ref{prod:1}
  \Leftrightarrow \ref{prod:2})$, it is enough to show that $W
  \subseteq D \times E$ is Scott open iff (a) for each $d \in D$ the
  set $V_d \eqdef \{ e \in D \mid (d,e) \in W \}$ is Scott open, and
  (b) for each Scott open set $\mathcal{V}$ of Scott open sets of $E$,
  the set $U_\mathcal{V} \eqdef \{ d \in D \mid V_d \in \mathcal{V}
  \}$ is Scott open.  We omit the somewhat long, but routine
  verification that this is the case.
\end{proof}

Here is a another argument that side-steps Theorem~\ref{prod:charac} but
uses the same ingredients as its proof:
\begin{proof}
  Let $D$ and~$E$ be two dcpos.  Write $\mathcal{A}(A,B)$ to denote
  the hom-set of a pair $A,B$ of objects of a category $A$, and
  $\mathcal{A}[A,B]$ to denote the exponential $B^A$ if it exists.
  Then, regarding $\Sierp$ both as a ($\mathcal{C}$-)space
  and a dcpo by an abuse of notation, we calculate, using obvious
  canonical isomorphisms:
  \begin{eqnarray*}
  \O (\Scott D \times_\mathcal{C} \Scott E) & \cong &
  \mathrm{Top}_\mathcal{C}(\Scott D \times_\mathcal{C} \Scott
  E,\Sierp) \\
  & \cong  &\mathrm{Top}_\mathcal{C}(\Scott D, \mathrm{Top}_\mathcal{C}[\Scott E,\Sierp]) \\
  & \cong  &\mathrm{Top}_\mathcal{C}(\Scott D, \Scott \mathrm{DCPO}[E,\Sierp]) \quad \text{by~\cite[Corollary
 5.16]{escardo:lawson:simpson}}\\
  & \cong &  \mathrm{DCPO}(D,\mathrm{DCPO}[E,\Sierp]) \\
  &  \cong & \mathrm{DCPO}(D \times_\mathrm{DCPO} E,\Sierp) \\
  & \cong & \O \Scott(D \times_\mathrm{DCPO} E).
 \end{eqnarray*}
 Moreover, the composition of all the canonical isomorphisms is easily
 seen to be the identity, because the transpositions are calculated as
 in the category of sets, and hence $\O (\Scott D \times_\mathcal{C}
 \Scott E) = \O \Scott(D \times_\mathrm{DCPO} E)$. Because both
 products are set-theoretical products with appropriate structure, we
 conclude that $\Scott D \times_\mathcal{C} \Scott E = \Scott(D
 \times_\mathrm{DCPO} E)$, as required.
\end{proof}

As a further corollary we obtain the known fact that the restriction
of the functor $\Scott \colon \mathrm{DCPO} \to \mathrm{Top}$ to
continuous dcpos preserves finite products. The reason is that
continuous dcpos are core-compact in the Scott topology, and hence are
in the class~$\mathcal{E}$, and that $X \times_{\mathcal{E}} Y = X
\times Y$ if one of the factors is in $\mathcal{E}$.  Moreover, this
argument establishes, more generally, the following fact, which is
also known~\cite{gierz:domains}:
\begin{corollary}
  The restriction of the functor $\Scott \colon \mathrm{DCPO} \to
  \mathrm{Top}$ to dcpos that are core-compact in their Scott
  topology preserves finite products.
\end{corollary}

\clearpage

\bibliographystyle{plain}
{\bibliography{references}}

\begin{thebibliography}{1}

\bibitem{bourbaki:topology1}
N.~Bourbaki.
\newblock {\em General Topology}, volume~1.
\newblock Addison-Wesley, London, 1988.

\bibitem{escardo:barbados}
M.H. Escard\'o.
\newblock Synthetic topology of data types and classical spaces.
\newblock {\em Electron. Notes Theor. Comput. Sci.}, 87:21--156, 2004.

\bibitem{escardo:heckmann:functionspace}
M.H. Escard\'o and R.~Heckmann.
\newblock Topologies on spaces of continuous functions.
\newblock {\em Topology Proceedings}, 26(2):545--564, 2001--2002.

\bibitem{escardo:lawson:simpson}
M.H. Escard{\'o}, J.~Lawson, and A.~Simpson.
\newblock Comparing {C}artesian closed categories of (core) compactly generated
  spaces.
\newblock {\em Topology Appl.}, 143(1-3), 2004.

\bibitem{gierz:domains}
G.~Gierz, K.H. Hofmann, K.~Keimel, J.D. Lawson, M.~Mislove, and D.S. Scott.
\newblock {\em Continuous Lattices and Domains}.
\newblock Cambridge University Press, 2003.

\bibitem{nachbin:compactinter}
L.~Nachbin.
\newblock Compact unions of closed subsets are closed and compact intersections
  of open subsets are open.
\newblock {\em Portugal. Math.}, 49(4):403--409, 1992.

\end{thebibliography}

\medskip

\paragraph{Links to the original versions of this paper:}
 \href{https://www.cs.bham.ac.uk/~mhe/papers/compactintersection.pdf}{2000},
 \href{https://www.cs.bham.ac.uk/~mhe/papers/compactness.pdf}{2005},
 \href{https://www.cs.bham.ac.uk/~mhe/papers/compactness-submitted.pdf}{2009 (this version)}.

\end{document}